\documentclass[runningheads]{llncs}

\usepackage{kotex}
\usepackage{amsfonts}
\usepackage{amsmath}
\usepackage{amssymb}
\usepackage{xcolor}
\usepackage{graphicx}
\usepackage{mathtools}
\usepackage{enumitem}
\usepackage[bookmarks=false]{hyperref}
\usepackage[misc]{ifsym}
\usepackage{todonotes}

\newcommand{\Z}{{\mathbb Z}}

\begin{document}

\title{On the Scaled Inverse of $(x^i-x^j)$ \\ modulo Cyclotomic Polynomial\\ of the form $\Phi_{p^s}(x)$ or $\Phi_{p^s q^t}(x)$}
\titlerunning{On the Scaled Inverse of $(x^i-x^j)$}
\author{Jung Hee Cheon
\inst{1, 4}
\and
Dongwoo Kim
\inst{2}\thanks{Work done while at Seoul National University.}
\and
Duhyeong Kim
\inst{3}\thanks{Work done while at Seoul National University.}
\and
Keewoo Lee$^{(\textnormal{\Letter})}$
\inst{1}
}

\authorrunning{J. H. Cheon et al.}
\institute{Seoul National University, Seoul, Republic of Korea\\
\email{\{jhcheon, activecondor\}@snu.ac.kr} \and
Western Digital Research, Milpitas, USA\\
\email{Dongwoo.Kim@wdc.com} \and
Intel Labs, Hillsboro, USA\\
\email{duhyeong.kim@intel.com} \and
Crypto Lab Inc., Seoul, Republic of Korea} 
\maketitle              %
\begin{abstract}
The scaled inverse of a nonzero element $a(x)\in \mathbb{Z}[x]/f(x)$, where $f(x)$ is an irreducible polynomial over $\mathbb{Z}$, is the element $b(x)\in \mathbb{Z}[x]/f(x)$ such that $a(x)b(x)=c \pmod{f(x)}$ for the smallest possible positive integer scale $c$. In this paper, we investigate the scaled inverse of $(x^i-x^j)$ modulo cyclotomic polynomial of the form $\Phi_{p^s}(x)$ or $\Phi_{p^s q^t}(x)$, where $p, q$ are primes with $p<q$ and $s, t$ are positive integers. Our main results are that the coefficient size of the scaled inverse of $(x^i-x^j)$ is bounded by $p-1$ with the scale $p$ modulo $\Phi_{p^s}(x)$, and is bounded by $q-1$ with the scale not greater than $q$ modulo $\Phi_{p^s q^t}(x)$. Previously, the analogous result on cyclotomic polynomials of the form $\Phi_{2^n}(x)$ gave rise to many lattice-based cryptosystems, especially, zero-knowledge proofs. Our result provides more flexible choice of cyclotomic polynomials in such cryptosystems. Along the way of proving the theorems, we also prove several properties of $\{x^k\}_{k\in\mathbb{Z}}$ in $\mathbb{Z}[x]/\Phi_{pq}(x)$ which might be of independent interest. 
\end{abstract}

\section{Introduction}

Cyclotomic polynomials play an important role in algebra, number theory, combinatorics, and their applications. 
In particular, modern lattice-based cryptography intensively employs cyclotomic rings $\Z[x]/\Phi_M(x)$ \cite{LPR10,LPR13}.

An interesting subclass of cyclotomic polynomials is of the form $\Phi_{p^s q^t}(x)$, where $p, q$ are primes with $p<q$ and $s, t$ are positive integers. 
Since cyclotomic polynomials of the form $\Phi_{p^s}(x)$ are just $\sum_{i=0}^{p-1}x^{ip^{s-1}}$, the case with two prime factors can be seen as the simplest non-trivial case. There have been various interesting results on these cyclotomic polynomials \cite{Bei64,HLLP12,Fou13}. For instance, these cyclotomic polynomials have only $\{-1,0,1\}$ as coefficients, whereas a cyclotomic polynomial of a product of three distinct odd primes can have an arbitrarily large coefficient \cite{Leh36}.

Benhamouda et. al. \cite{BCK+14} provided the following lemma, which was used to construct more efficient zero-knowledge proofs for lattice-based cryptosystems. The construction is being widely used \cite{BCS19,CKR+20}.
\begin{lemma}[\cite{BCK+14}] \label{lem:2s}
Let $M=2^s$ be a power-of-two. For any $i,j \in \Z$ satisfying $0\leq j<i<M$, there exists $u(x) \in \Z[x]/\Phi_{M}(x)$ such that \begin{itemize}[label=$\bullet$]
        \item $(x^i - x^j) \cdot u(x) = 2 \pmod{\Phi_{M}(x)}$
        \item and $||u(x)||_{\infty} \le 1$.
    \end{itemize}
\end{lemma}
Later, Lemma~\ref{lem:2s} was extended to the case of $M$ being a prime $p$ \cite{CKL21}.

In this paper, we generalize these phenomena as the \emph{scaled inverses} modulo cyclotomic polynomials (Definition~\ref{def:scaled_inv}).  
The scaled inverse of a nonzero element $a(x)\in \Z[x]/f(x)$, where $f(x)$ is an irreducible polynomial over $\Z$, is the element $b(x)\in \Z[x]/f(x)$ such that $a(x)b(x)=c \pmod{f(x)}$ for the smallest possible positive integer scale $c$. 
We investigate the scaled inverse of $(x^i-x^j)$ modulo cyclotomic polynomials of the form $\Phi_{p^s}(x)$ or $\Phi_{p^s q^t}(x)$. 

First, we generalize the previous results \cite{BCK+14,CKL21} to $\Phi_{p^s}(x)$ case: the coefficient size of the scaled inverse of $(x^i-x^j)$ is bounded by $p-1$ with the scale $p$ modulo $\Phi_{p^s}(x)$ (Theorem~\ref{thm:ps}). 
Second, we extend the results to $\Phi_{p^s q^t}(x)$ case: the coefficient size of the scaled inverse of $(x^i-x^j)$ is bounded by $q-1$ with the scale not greater than $q$ modulo $\Phi_{p^s q^t}(x)$ (Theorem~\ref{thm:psqt2}).

Our results have applications in cryptography. For instance, they are closely related to the efficiency and quality of certain zero-knowledge proofs regarding lattice-based cryptosystems with $\Z_{2^k}$-messages~\cite{CKL21}.\footnote{Utilization of $\Phi_{2^n}(x)$ cyclotomic rings gives much worse efficiency in this case.}

Along the way of proving the theorems, we prove several properties of $\{x^k\}_{k\in\Z}$ in $\Z[x]/\Phi_{pq}(x)$ which might be of independent interest (Section~\ref{sec:prop_pq}). We also investigate so-called \emph{expansion factors} of $\{x^k\}_{k\in\Z}$ in $\Z[x]/\Phi_{p^s}(x)$ and $\Z[x]/\Phi_{p^s q^t}(x)$, which also play important roles in zero-knowledge proofs regarding lattice-based cryptosystems \cite{CKL21}. The \emph{expansion factor} of $f(x)$ in $\Z[x]/\Phi_M(x)$ is defined as the maximum value of $(||f(x)\cdot g(x)||_\infty / ||g(x)||_\infty)$ (Section~\ref{sec:expansion}).

\section{Preliminaries}

\subsection{Notations}

In this subsection, we list notations which we will use throughout the paper, especially the ones which might be ambiguous to some readers. 

\begin{itemize}
    \item Throughout the paper, $p, q$ are primes satisfying $p<q$, and $s, t$ are positive integers, even if they are not explicitly mentioned.
    \item We denote the $M$th cyclotomic polynomial as $\Phi_M(x)$ and denote the Euler's totient function as $\phi(\cdot)$, i.e. $\phi(M) = \deg \Phi_M(x)$.
    \item We carefully distinguished the use of ``$\bmod{\Phi_M(x)}$''\footnote{\textsf{bmod} in \LaTeX} and ``$\pmod{\Phi_M(x)}$''\footnote{\textsf{pmod} in \LaTeX}. 
    We use ``$\bmod{\Phi_M(x)}$'' as a unary function which reduces the input polynomial modulo $\Phi_M(x)$ so that the degree of the output is less than $\phi(M)$. On the other hand, we use ``$\pmod{\Phi_M(x)}$'' to express a certain equality holds for residue classes under $\Phi_M(x)$. For example, ``$a(x)=b(x)\bmod{\Phi_M(x)}$'' says that, when $b(x)$ is fully reduced modulo $\Phi_M(x)$, the result is exactly equal to $a(x)$ as a polynomial in $\Z[x]$. On the contrary, ``$a(x)=b(x) \pmod{\Phi_M(x)}$'' says that $a(x)$ and $b(x)$ belong to same residue class under $\Phi_M(x)$.  
    \item We define the maximum norm $||\cdot||_\infty$ of $f(x)\in\Z[x]$ as the largest absolute value of coefficients of $f(x)$. We define the maximum norm $||\cdot||_\infty$ of $g(x)\in\Z[x]/\Phi_M(x)$ as the largest absolute value of coefficients of $(\tilde{g}(x)\bmod{\Phi_M(x)})$, where $\tilde{g}(x)\in \Z[x]$ is a representative of $g(x)$.
    \item For a polynomial $a(x)\in\Z[x]$, we denote the \emph{reverse polynomial} of $a(x)$ as $\mathsf{rev}(a(x))$, i.e. $\mathsf{rev}(a(x))=x^{\deg(a(x))}\cdot a(1/x)$.
    \item We denote the interval $\{i\in\Z|c \leq i \leq d\}$ as $[c,d]$.
    \item We denote the greatest common divisor of $a$ and $b$ as $(a,b)$.
\end{itemize}

\subsection{Properties of Cyclotomic Polynomials}

In this subsection, we recall and give short proofs on properties of cyclotomic polynomials $\Phi_p(x)$ and $\Phi_{pq}(x)$, which will be frequently used in the remaining parts of this paper. We only assume knowledge on basics of cyclotomic polynomials and very light knowledge on generating functions.

\begin{lemma}\label{lem:properties}
~
\begin{enumerate}[label=\textnormal{(\alph*)}]
    \item $\Phi_{pq}(1)=1$, i.e. $\frac{\Phi_{pq}(x)-1}{x-1}$ is a polynomial.
    \item $\Phi_{pq}(x)$ is symmetric, i.e. $\mathsf{rev}(\Phi_{pq}(x))=\Phi_{pq}(x).$
\end{enumerate}
\end{lemma}
\begin{proof}
~
\begin{enumerate}[label=\textnormal{(\alph*)}]
    \item Since $\Phi_{p}(x) \cdot \Phi_{q}(x) \cdot \Phi_{pq}(x) = \sum_{i=0}^{pq-1}x^i$ holds,  $p\cdot q \cdot \Phi_{pq}(1)=pq$.
    \item $\Phi_{pq}(x)$ can be written as $(\sum_{i=0}^{pq-1}x^i)/(\Phi_p(x)\cdot\Phi_q(x))$. Since $\Phi_p(x)\cdot\Phi_q(x)$ is symmetric, $\Phi_{pq}(x)$ is a quotient of symmetric polynomials, where the denominator divides the divisor. \qed
\end{enumerate}
\end{proof}

\begin{lemma}\label{lem:phi_pq_diff_coeff}
Denote the $i$th coefficient of $\frac{\Phi_{pq}(x)-1}{x-1}$ as $b_i$, i.e. $\frac{\Phi_{pq}(x)-1}{x-1} =\sum_i b_i\cdot x^i$.
Then, we can characterize $b_i$ as follows. 
$$b_i=
\begin{cases}
0 &\text{if } \alpha p + \beta q = i \text{ has a non-negative integer solution } (\alpha,\beta). \\
1 &\text{otherwise.}
\end{cases}$$
\end{lemma}
\begin{proof} The lemma follows from the following equalities.
\begin{align*} 
\frac{\Phi_{pq}(x)-1}{x-1}
&= \frac{1}{1-x}-\frac{x^{pq}-1}{x^p-1}\cdot \frac{1}{1-x^q} \\ 
&= (1+x+\cdots) - (1+x^p+x^{2p}+\cdots+x^{pq-p}) \cdot (1+x^q+x^{2q}+\cdots) 
\end{align*}
\qed
\end{proof}

\begin{corollary}\label{cor:phi_pq_diff_coeff}
~
\begin{enumerate}[label=\textnormal{(\alph*)}]
    \item If $p$ divides $i$, the $i$th coefficient of $\frac{\Phi_{pq}(x)-1}{x-1}$ is $0$. 
    \item For $0 \leq i < q$, the $i$th coefficient of $\frac{\Phi_{pq}(x)-1}{x-1}$ is $0$ if and only if $p$ divides $i$. 
    \item For $i\geq \phi(pq)$, $\alpha\cdot p + \beta \cdot q = i$ has a non-negative integer solution $(\alpha,\beta)$.
    \item For $0\leq i \leq \phi(pq)-1$, one of the $i$th and $(\phi(pq)-i-1)$th coefficients of $\frac{\Phi_{pq}(x)-1}{x-1}$ is $0$ and the other is $1$.
    
\end{enumerate}
\end{corollary}
\begin{proof}
~
\begin{enumerate}[label=\textnormal{(\alph*)}]
    \item The equation $\alpha\cdot p + \beta \cdot q = t \cdot p$ has a non-negative integer solution $(t,0)$. 
    \item If $\beta$ is positive, $\alpha\cdot p + \beta \cdot q \geq q$ holds. Therefore, for $0 \leq i < q$, the equation $\alpha\cdot p + \beta \cdot q = i$ has a non-negative integer solution $(\alpha,\beta)$ if and only if $p$ divides $i$.
    \item This follows from the fact that $\deg(\frac{\Phi_{pq}(x)-1}{x-1})=\phi(pq)-1$.
    \item From Lemma~\ref{lem:properties} (b), the following equalities hold. Then, recall Lemma~\ref{lem:phi_pq_diff_coeff}.
    \begin{align*} 
    \frac{\Phi_{pq}(x)-1}{x-1} + \mathsf{rev}\left(\frac{\Phi_{pq}(x)-1}{x-1}\right)
    &= \frac{\Phi_{pq}(x)-1}{x-1} + x^{\phi(pq)-1}\cdot \left(\frac{\Phi_{pq}(1/x)-1}{1/x-1}\right) \\ 
    &= \frac{\Phi_{pq}(x)-1}{x-1} + \frac{\Phi_{pq}(x)-x^{\phi(pq)}}{1-x} \\
    &= \frac{x^{\phi(pq)}-1}{x-1}
    \end{align*}
\end{enumerate}
\qed
\end{proof}

\section{Scaled Inverse of $(x^i-x^j)$ modulo $\Phi_{p^s}(x)$}

In this section, we prove Theorem \ref{thm:ps} regarding the scaled inverse of $(x^i-x^j)$ modulo $\Phi_{p^s}(x)$. Beforehand, we define the scaled inverse, and check its basic properties.

\subsection{Scaled Inverse}

\begin{definition}[Scaled Inverse] \label{def:scaled_inv}
Let $f(x)$ be an irreducible polynomial over $\Z$.
The scaled inverse of a nonzero element $a(x)\in \Z[x]/f(x)$ is the element $b(x)\in \Z[x]/f(x)$ such that $a(x)b(x)=c \pmod{f(x)}$ for the smallest possible positive integer $c$. We say $b(x)$ is the scaled inverse of $a(x)$ modulo $f(x)$ with scale $c$.  
\end{definition}

\begin{remark}[Existence]\label{rem:scaled_inv_exist}
Let $\check{a}(x)\in \Z[x]$ be the representative of $a(x)$ where $\deg(\check{a})<\deg(f)$.
Note that $(\check{a}(x), f(x))=1$, since $f(x)$ is irreducible. Thus, the resultant $r:=\mathsf{res}(\check{a}(x), f(x))$ is a nonzero integer. There exist Bezout coefficients $s(x), t(x)\in \Z[x]$ such that $s(x)\check{a}(x)+t(x)f(x)=r$, $\deg(s)<\deg(f)$, and $\deg(t)<\deg(\check{a})$. 
Thus, there exists a scaled inverse with scale not greater than $r$.
\end{remark}

\begin{remark}[Uniqueness]
Note the uniqueness of Bezout coefficients $\tilde{s}(x), \tilde{t}(x)\in \mathbb{Q}[x]$ such that $\tilde{s}(x)\check{a}(x)+\tilde{t}(x)f(x)=1$, $\deg(\tilde{s})<\deg(f)$, and $\deg(\tilde{t})<\deg(\check{a})$. 
Followingly, $\left(c \tilde{s}(x) \bmod{f(x)}\right)$ is the unique scaled inverse. 
\end{remark}

\begin{remark}[Formulation]\label{rem:scaled_inv_formula}
Let $\mathsf{cont}(s)$ be the positive content of $s(x)$. Let $d$ be $(r, \mathsf{cont}(s))$. Then, it is easy to see that $b(x):=s(x)/d \bmod{f(x)}$ is the scaled inverse with scale $c:=r/d$.
\end{remark}

\subsection{Proof of Theorem~\ref{thm:ps}}

\begin{theorem}\label{thm:ps}
Let $p$ be a prime and $M=p^s$ be a prime power. For any $i,j \in \Z$ satisfying $0\leq j<i<M$, there exists $u(x) \in \Z[x]/\Phi_{M}(x)$ such that \begin{itemize}[label=$\bullet$]
        \item $(x^i - x^j) \cdot u(x) = p \pmod{\Phi_{M}(x)}$
        \item and $||u(x)||_{\infty} \le p-1$.
    \end{itemize}
\end{theorem}

Theorem~\ref{thm:ps} says the coefficient size of the scaled inverse of $(x^i-x^j)$ is bounded by $p-1$ with the scale $p$ modulo $\Phi_{p^s}(x)$. 
Regarding Remark~\ref{rem:scaled_inv_formula}, $u(x)$ is indeed the scaled inverse: coefficients of $u(x)$ is already smaller than $p$, which is the only non-identity factor of $p$.

For readers' comprehension, we first review the proof of $s=1$ case which was previously presented in \cite{CKL21}. The full proof of Theorem~\ref{thm:ps} is a straightforward generalization of the $s=1$ case. However, the full proof requires unpleasant notations and computations. Readers might want to first read the $s=1$ case and catch the outline of the full proof.

\begin{proof}[\cite{CKL21} $s=1$ Case]
Consider the following polynomial in $\Z[x]$.
$$v(x):=\frac{\Phi_p(x)-p}{x-1}=\sum_{k=0}^{p-1}(p-1-k)\cdot x^k$$
We claim that $\tilde{u}(x):=-x^{p-j}\cdot v(x^{i-j}) \in \Z[x]$ satisfies the conditions after being reduced by $\Phi_{p}(x)$. By definition, the first condition can be easily checked with the fact that $\Phi_{p}(x)$ divides $\Phi_{p}(x^{i-j})$ since $(p,i-j)=1$. 

Since $p$ does not divide $i-j$, when reduced modulo $x^{p}-1$, each monomials of $\tilde{u}(x)$ are reduced to distinct-degree monomials with coefficients remaining in the interval $[1-p,0]$. Let us denote the $\ell$th coefficient of $(\tilde{u}(x) \bmod{x^{p}-1})$ as $\tilde{u}_\ell\in[1-p,0]$. 
Applying modulo $\Phi_{p}(x)$ to $(\tilde{u}(x) \bmod{x^{p}-1})$, the $\ell$th coefficients of $(\tilde{u}(x) \bmod{\Phi_{p}(x)})$ equals $\tilde{u}_\ell-\tilde{u}_{(p-1)}$.
Certainly, $\tilde{u}_\ell-\tilde{u}_{(p-1)}$ lies in the interval of $[1-p,p-1]$. 
Thus, the inequality $||\tilde{u}(x) \bmod{\Phi_{p}(x)}||_\infty \leq p-1$ holds. \qed
\end{proof}

Now we give the actual proof of Theorem~\ref{thm:ps} for arbitrary $s$.

\begin{proof}[Theorem~\ref{thm:ps}]
Let $p^\alpha$ be the largest power of $p$ dividing $i-j$, and let $\beta:=(i-j)/p^\alpha$. Let us denote $M'=p^{s-1}$.
Consider the following polynomial $v(x) \in \Z[x]$.
\begin{align*} 
v(x) &:= \frac{\Phi_{M}(x^\beta)-p}{x^{p^{\alpha}\beta}-1} \\ 
&=  \frac{\Phi_{p}(x^{M'\beta})-p}{x^{M'\beta}-1} \cdot \frac{x^{M'\beta}-1}{x^{p^{\alpha}\beta}-1}\\
&= \sum^{p-1}_{k=0} (p-1-k) \left[ x^{(M' k) \beta} + x^{(M' k + p^{\alpha}) \beta} + \cdots +  x^{(M' k + M' - p^{\alpha}) \beta} \right]
\end{align*}
We claim that $\tilde{u}(x)=-x^{M-j}\cdot v(x) \in \Z[x]$ satisfies the conditions after being reduced by $\Phi_{M}(x)$.
By definition, the first condition can be easily checked with the fact that $\Phi_{M}(x)$ divides $\Phi_{M}(x^\beta)$ since $(M,\beta)=1$. 

For the second condition, first observe that the degrees of monomials with nonzero coefficients in $\tilde{u}(x)$ are same modulo $p^\alpha \beta$.
Moreover, the coefficients of $\tilde{u}(x)$ are in the interval of $[1-p, 0]$.
Since $(M,\beta)=1$, when reduced modulo $x^{M}-1$, each monomials of $\tilde{u}(x)$ are reduced to distinct-degree monomials (degrees being same modulo $p^\alpha$) with coefficients remaining in the interval of $[1-p,0]$. Let us denote the $\ell$th coefficient of $(\tilde{u}(x) \bmod{x^{M}-1})$ as $\tilde{u}_\ell\in[1-p,0]$. 
Applying modulo $\Phi_{M}(x)$ to $(\tilde{u}(x) \bmod{x^{M}-1})$, the $\ell$th coefficients of $(\tilde{u}(x) \bmod{\Phi_{M}(x)})$ equals $\tilde{u}_\ell-\tilde{u}_m$, where $m$ is the largest integer which equals $\ell$ modulo $M'$ and less than $M$. 
Certainly, $\tilde{u}_\ell-\tilde{u}_m$ lies in the interval of $[1-p,p-1]$. 
Thus, the inequality $||\tilde{u}(x) \bmod{\Phi_{M}(x)}||_\infty \leq p-1$ holds. \qed
\end{proof}
We remark that Theorem~\ref{thm:ps} is tight: when $i=1$ and $j=0$, the $0$th coefficient of $u(x)$ is $p-1$ and followingly $||u(x)||_\infty = p-1$.

\section{Properties of $\{x^k\}_{k\in\Z}$ modulo $\Phi_{pq}(x)$} \label{sec:prop_pq}

In this section, we prove several properties of $\{x^k\}_{k\in\Z}$ in $\Z[x]/\Phi_{pq}(x)$. These results are not only the essence of the proof of Theorem \ref{thm:psqt} in Section~\ref{sec:scinverse}, but also could be of independent interest.

\subsection{Properties of $\{x^k\}_{k\in\Z}$ modulo $\Phi_{pq}(x)$}

\begin{lemma}\label{lem:exp_b_1}
The following equalities hold for $0 \leq k \leq p-1$.
\begin{align*}
x^{\phi(pq)+k}\bmod{\Phi_{pq}(x)} &= x^{\phi(pq)+k} - \Phi_{pq}(x) \cdot \sum^{k}_{i=0}x^i \\
||x^{\phi(pq)+k} \bmod{\Phi_{pq}(x)}||_\infty &= 1
\end{align*}
\end{lemma}
\begin{proof}
Let us denote the $j$th coefficient of $\Phi_{pq}(x) \cdot \sum^{k}_{i=0}x^i$ as $d_j$, i.e. $\Phi_{pq}(x) \cdot \sum^{k}_{i=0}x^i=\sum_j d_j\cdot x^j$. Consider the following representation. 
\begin{align*} 
\Phi_{pq}(x) \cdot \sum^{k}_{i=0}x^i
&= \frac{x^{pq}-1}{(x-1)\cdot \Phi_p(x) \cdot \Phi_q(x)}\cdot \frac{x^{k+1}-1}{x-1} \\ 
&= \frac{x^{pq}-1}{x^p-1}\cdot \frac{1-x^{k+1}}{1-x^q} \\
&= (1+x^p+x^{2p}+\cdots+x^{(q-1)p})\cdot(1-x^{k+1})\cdot(1+x^q+x^{2q}+\cdots) 
\end{align*}
Now we can interpret $d_i$'s combinatorially. That is, for Diophantine equations 
\begin{align}
    \alpha p + \beta q &= i \\
    \alpha p + \beta q &= i - (k+1),
\end{align}
$$d_i=
\begin{cases}
1 &\text{if (1) has a non-negative integer solution } (\alpha,\beta) \text{ but (2) does not.}\\
-1 &\text{if (2) has a non-negative integer solution } (\alpha,\beta) \text{ but (1) does not.}\\
0 &\text{otherwise.}
\end{cases}$$
Therefore, we proved that $||x^{\phi(pq)+k} - \Phi_{pq}(x) \cdot \sum^{k}_{i=0}x^i||_\infty = 1$.

Equation (1) has a non-negative integer solution for $\phi(pq) \leq i \leq \phi(pq)+k$ (Corollary \ref{cor:phi_pq_diff_coeff} (c)). On the other hand, the equation (2) has non-negative integer solutions for any $\phi(pq) \leq i < \phi(pq)+k$ (Lemma \ref{lem:phi_pq_diff_coeff}, Corollary \ref{cor:phi_pq_diff_coeff} (b),(d)). Furthermore, it is easy to see that equation (2) has no solution for $i=\phi(pq)+k$, since $\deg(\frac{\Phi_{pq}(x)-1}{x-1})=\phi(pq)-1$ (Lemma \ref{lem:phi_pq_diff_coeff}). Together with the combinatorial characterization of $d_i$, $d_i=0$ holds for $\phi(pq) \leq i < \phi(pq)+k$ and $d_i=1$ holds for $i=\phi(pq)+k$. Then, the lemma follows.
\qed
\end{proof}

\begin{corollary}\label{cor:lowest_coeff_is}
For $0\leq k <p-1$, the $0$th coefficient of $(x^{\phi(pq)+k}\bmod{\Phi_{pq}(x)})$ equals $-1$.
\end{corollary}
\begin{proof}
The corollary follows from Lemma \ref{lem:exp_b_1} and the fact that $\Phi_{pq}(0)=1$ by Lemma \ref{lem:properties} (b). \qed
\end{proof}

\begin{corollary}\label{cor:highest_coeff_is_one}
For $0\leq k <p-1$, the $(\phi(pq)-1)$th coefficient of $(x^{\phi(pq)+k}\bmod{\Phi_{pq}(x)})$ is $1$.
\end{corollary}
\begin{proof}
Considering the following equalities, the corollary follows from Lemma \ref{lem:exp_b_1} and Corollary \ref{cor:phi_pq_diff_coeff} (b),(d).
\begin{align*} 
\Phi_{pq}(x) \cdot \sum^{k}_{i=0}x^i
&= \left( \frac{\Phi_{pq}(x)-1}{x-1} \cdot (x-1) +1 \right)\cdot \sum^{k}_{i=0}x^i \\ 
&=  \frac{\Phi_{pq}(x)-1}{x-1} \cdot (x^{k+1}-1) +  \sum^{k}_{i=0}x^i 
\end{align*}
\qed
\end{proof}

\begin{lemma}\label{lem:exp_b_2}
The following equality holds for $p-1\leq k \leq q-1$.
$$x^{\phi(pq)+k}\bmod{\Phi_{pq}(x)} = -x^{k-(p-1)}\sum^{p-2}_{i=0}x^{q\cdot i}$$
\end{lemma}
\begin{proof}
The lemma directly follows from the following equalities. The first equality is from Lemma \ref{lem:exp_b_1}.
$$x^{\phi(pq)+p-1}\bmod{\Phi_{pq}(x)} = x^{pq-q}-\Phi_{pq}(x)\cdot \Phi_{p}(x)= x^{pq-q}-\frac{x^{pq}-1}{x^q-1}= -\sum^{p-2}_{i=0}x^{q\cdot i} $$ \qed
\end{proof}

\begin{lemma}\label{lem:exp_b_3}
The following equality holds for $0 \leq k < pq-\phi(pq)$.
$$\mathsf{rev}\left(x^{\phi(pq)+k}\bmod{\Phi_{pq}(x)}\right) = x^{pq-1-k}\bmod{\Phi_{pq}(x)}$$
\end{lemma}
\begin{proof}
Let $x^{\phi(pq)+k}\bmod{\Phi_{pq}(x)}=x^{\phi(pq)+k}-f(x)\cdot\Phi_{pq}(x)$. Note that $\deg(f)<pq-\phi(pq)$. By the symmetry of $\Phi_{pq}(x)$ (Lemma \ref{lem:properties} (b)), the following equalities hold.
\begin{align*} 
\mathsf{rev}\left(x^{\phi(pq)+k}\bmod{\Phi_{pq}(x)}\right) &=  \mathsf{rev}\left(x^{\phi(pq)+k}-f(x)\cdot\Phi_{pq}(x)\right) \\ 
&= x^{\phi(pq)-1}\left((1/x)^{\phi(pq)+k}-f(1/x)\cdot\Phi_{pq}(1/x)\right) \\
&= x^{-k-1}-\frac{f(1/x)}{x} \cdot \Phi_{pq}(x) \\
&= x^{pq} \left(x^{-k-1}-\frac{f(1/x)}{x} \cdot \Phi_{pq}(x)\right) \pmod{\Phi_{pq}(x)} \\
&= x^{pq-k-1}-\left(x^{pq-1} \cdot f(1/x) \right) \Phi_{pq}(x) \pmod{\Phi_{pq}(x)} \\
&= x^{pq-k-1} \pmod{\Phi_{pq}(x)}
\end{align*}\qed
\end{proof}

\begin{corollary}\label{cor:exp_norm}
For all integer $k$, $||x^k \bmod{\Phi_{pq}(x)}||_\infty = 1$ holds.
\end{corollary}
\begin{proof}
This follows from Lemma \ref{lem:exp_b_1}, \ref{lem:exp_b_2}, \ref{lem:exp_b_3}. \qed
\end{proof}

\subsection{Reduction Matrix}

For a clearer demonstration, we end this section with a numerical example and how our results apply to it. Before we proceed, we define the \emph{reduction matrix} $R_M$ of a cyclotomic polynomial $\Phi_M(x)$ as a $\phi(M)\times M$ matrix with $(i,j)$-element being the $i$th coefficient of $(x^j \bmod{\Phi_M(x)})$ for $0\leq i<\phi(M)$ and $0\leq j < M$. 

First, it is easy to see that $R_p=(I|\mathbf{-1})$, where $\mathbf{-1}$ denotes the column filled with $-1$. For example, not writing zeroes down,
$$ R_7 =
\left(
\begin{array}{c c c c c c | c}
+1 &   &   &   &   &   &  -1 \\
  & +1 &   &   &   &   &  -1 \\
  &   & +1 &   &   &   &  -1 \\
  &   &   & +1 &   &   &  -1 \\
  &   &   &   & +1 &   &  -1 \\
  &   &   &   &   & +1 &  -1 
\end{array}
\right).
$$

Regarding Lemma \ref{lem:exp_b_2}, \ref{lem:exp_b_3}, we can describe $R_{pq}$ as $(I|B_1|B_2|B_3)$, where $B_1,B_3 \in \Z^{\phi(pq)\times (p-1)}$ and $B_2 \in \Z^{\phi(pq)\times (q-p+1)}$. Lemma \ref{lem:exp_b_2} says that $B_2$ is a very structured Toeplitz matrix and Lemma \ref{lem:exp_b_3} says that $B_3$ is a 180$^{\circ}$ rotation of $B_1$. Corollary \ref{cor:exp_norm} says that every element of $R_{pq}$ is in $\{-1,0,1\}$, Corollary \ref{cor:lowest_coeff_is} says that every elements of the $0$th row of $B_1$ is $-1$, and Corollary \ref{cor:highest_coeff_is_one} says that every elements of the $(\phi(pq)-1)$th row of $B_1$ is $1$. We can check all these properties with $R_{3\cdot 7}$ which is listed below without zeroes written down.
$$ 
\left(
\begin{array}{c c c c c c c c c c c c | c c | c c c c c | c c}
+1 &   &   &   &   &   &   &   &   &   &   &   & -1 & -1 & -1 &    &    &    &    & +1 & +1 \\
  & +1 &   &   &   &   &   &   &   &   &   &   & +1 &    &    & -1 &    &    &    & -1 &    \\
  &   & +1 &   &   &   &   &   &   &   &   &   &    & +1 &    &    & -1 &    &    &    & -1 \\
  &   &   & +1 &   &   &   &   &   &   &   &   & -1 & -1 &    &    &    & -1 &    & +1 & +1 \\
  &   &   &   & +1 &   &   &   &   &   &   &   & +1 &    &    &    &    &    & -1 & -1 &    \\
  &   &   &   &   & +1 &   &   &   &   &   &   &    & +1 &    &    &    &    &    & -1 & -1 \\
  &   &   &   &   &   & +1 &   &   &   &   &   & -1 & -1 &    &    &    &    &    & +1 &    \\
  &   &   &   &   &   &   & +1 &   &   &   &   &    & -1 & -1 &    &    &    &    &    & +1 \\
  &   &   &   &   &   &   &   & +1 &   &   &   & +1 & +1 &    & -1 &    &    &    & -1 & -1 \\
  &   &   &   &   &   &   &   &   & +1 &   &   & -1 &    &    &    & -1 &    &    & +1  &   \\
  &   &   &   &   &   &   &   &   &   & +1 &   &    & -1 &    &    &    & -1 &    &    & +1 \\
  &   &   &   &   &   &   &   &   &   &   & +1 & +1 & +1 &    &    &    &    & -1 & -1 & -1  
\end{array}
\right)
$$

\begin{remark}\label{rem:tensor}
Let $\bar{M}$ be the largest square-free divisor of $M$ and $M'=M/\bar{M}$. 
Since $\Phi_M(x) = \Phi_{\bar{M}}(x^{M'})$, $R_M$ equals $R_{\bar{M}} \otimes I$, where $\otimes$ denotes the Kronecker product.
In particular, $R_{p^s} = R_p \otimes I$ and $R_{p^s q^t} = R_{pq} \otimes I$ hold.  
\end{remark}

\subsection{Patterns in Reduction Matrix}

\begin{lemma}\label{lem:exp_sum_zero}
The following inequality holds.
$$\sum_{i=0}^{q-1}(x^{j+ip} \bmod{\Phi_{pq}(x)}) = 0$$
\end{lemma}
\begin{proof}
\begin{align*} 
\sum_{i=0}^{q-1}(x^{j+ip} \bmod{\Phi_{pq}(x)}) &=  x^j \cdot \sum_{i=0}^{q-1}x^{ip} \pmod{\Phi_{pq}(x)} \\ 
&= x^j \cdot \Phi_{pq}(x) \cdot \Phi_{q}(x) \pmod{\Phi_{pq}(x)} \\
&= 0 \pmod{\Phi_{pq}(x)}
\end{align*}
~\qed
\end{proof}

\begin{lemma}\label{lem:exp_nonzero_two}
For any $0 \leq k< \phi(pq)$ and $0 \leq j <p$, there are at most two $i$'s in $\{0,1,2, \cdots, q-1\}$ such that the $k$th coefficient of $(x^{j+ip} \bmod{\Phi_{pq}(x)})$ is nonzero.
\end{lemma}
\begin{proof}
Combining the facts that nonzero $k$th coefficients of $(x^{j+ip} \bmod{\Phi_{pq}(x)})$ are either $-1$ or $+1$ (Corollary \ref{cor:exp_norm}) and they sum up to zero (Lemma \ref{lem:exp_sum_zero}), it is sufficient to show that there are at most three $i$'s with nonzero $k$th coefficient of $(x^{j+ip} \bmod{\Phi_{pq}(x)})$.

Revisiting the reduction matrix $(I|B_1|B_2|B_3)$, at most one column corresponding to one of $\{x^{j+ip}\}_i$ may lie in each of $B_1,B_3\in \Z^{\phi(pq)\times (p-1)}$. In the $k$th row of the reduction matrix, $I$ has the only nonzero element at the coordinate of $(k,k)$ and $B_2$ may have nonzero element only at the coordinate of $(k,pq-q+k-q\lfloor k/q \rfloor)$ if this coordinate lies in $B_2$. However, $k\neq pq-q+k-q\lfloor k/q \rfloor \pmod{p}$ for any $0\leq k <\phi(pq)$. Hence, the lemma is proved.  \qed
\end{proof}

\begin{lemma}\label{lem:exp_sum_norm}
For any subset $I \subset \{0,1,2, \cdots, q-1\}$ and $0 \leq j <p$, the following inequality holds. 
$$\left|\left|\sum_{i\in I}\left(x^{j+ip} \bmod{\Phi_{pq}(x)} \right)\right|\right|_{\infty} \leq 1$$
\end{lemma}
\begin{proof}
From Lemma \ref{lem:exp_sum_zero}, \ref{lem:exp_nonzero_two}, and Corollary \ref{cor:exp_norm}, it is easy to see that the $k$th coefficients of $\{x^{j+ip} \bmod{\Phi_{pq}(x)}\}_{0 \leq i < q}$ are either all zero, or all zero except for one $+1$ and one $-1$.
Subset-sums of these sets are in ${-1,0,1}$. 
\qed
\end{proof}

\begin{corollary}
\label{cor:exp_sum_norm_psqt}
Let $M=p^sq^t$ and $M'=M/(pq)$ be integers where $s$ and $t$ are positive integers.
For any integer $0 \leq j <p$ and any family of subsets $I_k \subset \{0,1,2, \cdots, q-1\}$ on $0\leq k<M'$, the following inequality holds. 
$$\left|\left| \sum_{k=0}^{M'-1} \sum_{i\in I_k} \left( x^{(j+ip)M'+k} \bmod{\Phi_{M}(x)} \right)\right|\right|_{\infty} \leq 1$$
\end{corollary}

\begin{proof}
Since $\left( x^{(j+ip)M'+k} \bmod{\Phi_{M}(x)} \right)$ has nonzero coefficients only at the degrees those equal to $k$ modulo $M'$, terms with distinct $k$ do not interfere with each others.
Therefore, it is sufficient to prove the following inequality, which can be obtained from Lemma \ref{lem:exp_sum_norm} with $x$ substituted by $x^{M'}$.
$$\left|\left|  \sum_{i\in I_k} \left( x^{(j+ip)M'} \bmod{\Phi_{M}(x)} \right)\right|\right|_{\infty} \leq 1$$ \qed
\end{proof}

\section{Scaled Inverse of $(x^i-x^j)$ modulo $\Phi_{p^sq^t}(x)$}\label{sec:scinverse}

In this section, we prove Theorem \ref{thm:psqt} and \ref{thm:psqt2} regarding the scaled inverse of $(x^i-x^j)$ modulo $\Phi_{p^sq^t}(x)$. 
We begin with Theorem \ref{thm:psqt}. 
Theorem~\ref{thm:psqt} says the coefficient size of the (scaled) inverse of $(x^i-x^j)$ is bounded by $p-1$ modulo $\Phi_{p^s q^t}(x)$, if $p^s\nmid (i-j)$ and $q^t\nmid (i-j)$. 
The proof outline is similar to the proof of Theorem~\ref{thm:ps}. 
However, the details require the results in Section~\ref{sec:prop_pq}.

\begin{theorem}\label{thm:psqt}
Let $p$ and $q$ be primes satisfying $p<q$, and let $M=p^sq^t$ be an integer where $s$ and $t$ are positive integers.
For any integers $0\leq j<i<M$ satisfying $p^s\nmid (i-j)$ and $q^t\nmid (i-j)$, there exists $u(x) \in \Z[x]/\Phi_M(x)$ such that \begin{itemize}[label=$\bullet$]
        \item $(x^i - x^j) \cdot u(x) = 1 \pmod{\Phi_{M}(x)}$
        \item and $||u(x)||_{\infty} \le p-1$.
    \end{itemize}
\end{theorem}

\begin{proof}
Let $p^{\alpha}$ be the largest power of $p$ dividing $i-j$, let $q{^\beta}$ be the largest power of $q$ dividing $i-j$, and let $\gamma:=(i-j)/(p^\alpha q^\beta)$.
Let us denote $M'=p^{s-1} q^{t-1}$.
Consider the following polynomial $v(x) \in \Z[x]$. Note that $\alpha \leq s-1$ and $\beta \leq t-1$ by the assumption.
\begin{align*} 
v(x) &:= \frac{\Phi_{M}(x^\gamma)-1}{x^{p^{\alpha}q^{\beta}\gamma}-1} \\ 
&=  \frac{\Phi_{pq}(x^{M' \gamma})-1}{x^{M' \gamma}-1} \cdot \frac{x^{M' \gamma}-1}{x^{p^{\alpha}q^{\beta}\gamma}-1}
\end{align*}
We claim that $\tilde{u}(x)=-x^{M-j}\cdot v(x) \in \Z[x]$ satisfies the conditions after being reduced by $\Phi_{M}(x)$.
By definition, the first condition can be easily checked by the fact that $\Phi_{M}(x)$ divides $\Phi_{M}(x^\gamma)$ since $(M,\gamma)=1$.

For the second condition, first observe that the degrees of monomials with nonzero coefficients in $\tilde{u}(x)$ are same modulo $p^\alpha q^\beta \gamma$.
Moreover, the coefficients of $\tilde{u}(x)$ are either $-1$ or $0$ by Lemma \ref{lem:phi_pq_diff_coeff}.
Since $(M, \gamma)=1$, when reduced modulo $x^{M}-1$, each monomials of $\tilde{u}(x)$ are reduced to distinct-degree monomials (degrees being same modulo $p^\alpha q^\beta$) with coefficients remaining in $\{-1,0\}$. 

Since there are no monomials with a degree of multiple of $p M'$ in $v(x)$ (Corollary~\ref{cor:phi_pq_diff_coeff} (a)), we can group the monomials of $(\tilde{u}(x) \bmod{\Phi_M(x}))$ into $p-1$ classes according to the setting of Corollary \ref{cor:exp_sum_norm_psqt}. 
Then applying Corollary \ref{cor:exp_sum_norm_psqt} together with the triangle inequality, we are done. \qed
\end{proof}

We remark that Theorem~\ref{thm:psqt} is \emph{quite} tight according to the following lemma.\footnote{We remark that there are $M$'s whose $u(x)$ satisfy $||u(x)||_\infty \leq p-2$ for all $i$ and $j$ (e.g. 35). 
On the other hand, there are also $M$'s whose $u(x)$ satisfy $||u(x)||_\infty = p-1$ for some $i$ and $j$ (e.g. 33).} 

\begin{lemma}\label{lem:pq_lower}
For $u(x)$ defined in Theorem~\ref{thm:psqt} with $i=p^{s-1}q^{t-1}(p-1)$ and $j=p^{s-1}q^{t-1}(p-2)$, the following inequality holds.
$$||u(x)||_\infty \geq p-2$$
\end{lemma}

\begin{proof}
Using the proof of Theorem~\ref{thm:psqt} and the following equalities, we can reduce the general case to the $M=pq$ case with $s=1$ and $t=1$. 
\begin{align*} 
||u(x)||_\infty &= \left|\left| -x^{M-j}\cdot \frac{\Phi_{M}(x^\gamma)-1}{x^{p^{\alpha}q^{\beta}\gamma}-1} \bmod{\Phi_{M}(x) } \right|\right|_{\infty} \\
&= \left|\left| -y^{pq-(p-2)}\cdot \frac{\Phi_{pq}(y)-1}{y-1} \bmod{\Phi_{pq}(y) } \right|\right|_{\infty} \qquad (y=x^{p^{s-1}q^{t-1}})
\end{align*}

Now consider the following polynomial in $\Z[x]$. $$\tilde{u}(x)=\Phi_{pq}(x)+(p-1)\cdot \frac{\Phi_{pq}(x)-1}{x-1}-\sum^{p-1}_{i=1}\frac{x^{i\cdot q-p+2}-1}{x-1}$$
First, observe that $\deg(\tilde{u})\leq\phi(pq)-1$.
Then, by the following equalities, $\tilde{u}(x)$ satisfies the first condition of Theorem~\ref{thm:psqt} after being reduced by $\Phi_{pq}(x)$.
\begin{align*} 
\tilde{u}(x) \cdot (x^{p-1}-x^{p-2}) &=  \tilde{u}(x) \cdot x^{p-2} \cdot (x-1) \\ 
&= (p-1)\cdot x^{p-2}\cdot \left(\Phi_{pq}(x)-1\right)-\sum^{p-1}_{i=1}\left(x^{i\cdot q}-x^{p-2}\right) \\
&= -\sum^{p-1}_{i=1}x^{i\cdot q} \\
&= 1 \pmod{\Phi_{pq}(x)}
\end{align*}

Observe that the $0$th coefficient of $\tilde{u}(x)$ equals $-(p-2)$. This easily follows from the fact that $\Phi_{pq}(0)=1$. (Lemma \ref{lem:properties} (b)) 
Thus, $||\tilde{u}(x)||_\infty \geq p-2$. \qed

\end{proof}

Theorem~\ref{thm:psqt2} is an extension of Theorem~\ref{thm:psqt} with the help of Theorem~\ref{thm:ps}.
Theorem~\ref{thm:psqt2} says the coefficient size of the scaled inverse of $(x^i-x^j)$ is bounded by $q-1$ with the scale not greater than $q$ modulo $\Phi_{p^s q^t}(x)$. 
Regarding Remark~\ref{rem:scaled_inv_formula} and the proof of Theorem~\ref{thm:psqt2}, $u(x)$ is indeed the scaled inverse: coefficients of $u(x)$ is not divisible by the scale. 

\begin{theorem}\label{thm:psqt2}
Let $p$ and $q$ be primes satisfying $p<q$, and let $M=p^sq^t$ be an integer where $s$ and $t$ are positive integers.
For any integers $0\leq j<i<M$, there exists $u(x) \in \Z[x]/\Phi_M(x)$ such that \begin{itemize}[label=$\bullet$]
        \item $(x^i - x^j) \cdot u(x) = c \pmod{\Phi_{M}(x)}$,
        \item and $||u(x)||_{\infty} \le d$,
    \end{itemize}
    \quad where $(c, d)= \left\{ \begin{array}{ll}
        (q, q-1), & \qquad \text{if } p^s \mid (i-j)\\
        (p, p-1), & \qquad \text{if } q^t \mid (i-j)\\
        (1, p-1), & \qquad \text{otherwise. }
        \end{array} \right.$ 
\end{theorem}

\begin{proof}
If $p^s\nmid (i-j)$ and $q^t\nmid (i-j)$, use Theorem~\ref{thm:psqt} to get $u(x)$ with $(x^i - x^j) \cdot u(x) = 1 \pmod{\Phi_{M}(x)}$ and $||u(x)||_{\infty} \le p-1$.

If $p^s\mid (i-j)$, let $q{^\beta}$ be the largest power of $q$ dividing $i-j$, and let $\gamma:=(i-j)/(p^s q^\beta)$.
Consider the following polynomial $v(x) \in \Z[x]$. Note that $\beta \leq t-1$ by the assumption.
\begin{align*} 
v(x) &:= \frac{\Phi_{q^t}(x^{p^s \gamma})-q}{x^{p^s q^{\beta}\gamma}-1} \\ 
&=  \frac{\Phi_{q}(x^{p^{s} q^{t-1} \gamma})-q}{x^{p^{s} q^{t-1} \gamma}-1} \cdot \frac{x^{p^{s} q^{t-1} \gamma}-1}{x^{p^{s}q^{\beta}\gamma}-1}
\end{align*}
We claim that $\tilde{u}(x)=-x^{M-j}\cdot v(x) \in \Z[x]$ satisfies the conditions with $c=q$ after being reduced by $\Phi_{M}(x)$.
By definition, the first condition can be easily checked by the fact that $\Phi_{p^s q^t}(x)$ divides $\Phi_{q^t}(x^{p^s \gamma})$ since $p^s$, $q^t$, and $\gamma$ are mutually coprime. The second condition can be shown by the same argument in the proof of Theorem~\ref{thm:ps}.

If $q^t\mid (i-j)$, switch the role of $p$ and $q$ in the case of $p^s\mid (i-j)$. Then, we get $u(x)$ with $(x^i - x^j) \cdot u(x) = p \pmod{\Phi_{M}(x)}$ and $||u(x)||_{\infty} \le p-1$. \qed
\end{proof}

\section{Expansion Factors of $x^k$ modulo $\Phi_{p^s}(x)$ and $\Phi_{p^s q^t}(x)$} \label{sec:expansion}

In this section, we examine so-called \emph{expansion factors} of $\{x^k\}_{k\in\Z}$ in $\Z[x]/\Phi_{M}(x)$ with $M=p^s$ or $M=p^s q^t$. 
The \emph{expansion factor} of $f(x)$ in $\Z[x]/\Phi_M(x)$ is defined as the maximum value of $(||f(x)\cdot g(x)||_\infty / ||g(x)||_\infty)$ upon $g(x) \in \Z[x]/\Phi_M(x)$. 
The following lemmas say that the expansion factors of $\{x^k\}_{k\in\Z}$ modulo $\Phi_{p^s}(x)$ and $\Phi_{p^s q^t}(x)$ are not \emph{too large}. These lemmas are generalizations of the power-of-two case: the expansion factors of $\{x^k\}_{k\in\Z}$ modulo $\Phi_{2^s}(x)$ are 1, since multiplying $x^k$ in $\Z[x]/\Phi_{2^s}(x)$ acts as \emph{skewed}-rotation of the coefficients. 
The statements and proofs follow the framework of the $p$ case which is described in \cite{CKL21}. The results are also closely related to the quality of certain zero-knowledge proofs regarding lattice-based cryptosystems.

\begin{lemma} For $\mathcal{R}=\Z[x]/\Phi_{p^s}(x)$, the following equality holds.
$$\max_{\substack{ k \in \Z \\ g(x)\in \mathcal{R}}}\left\{\frac{||x^k\cdot g(x)||_\infty}{||g(x)||_\infty} \right\}=2$$
\end{lemma}
\begin{proof}
Consider the reduction matrix of $\Phi_{p^s}(x)$. Since any row of the matrix has two nonzero elements and they are either $-1$ or $+1$, $||x^k \cdot g(x)||_\infty \leq 2\cdot ||g(x)||_\infty$ holds for all $0\leq k<p^s$. Thus, the expansion factors of $\{x^k\}_{k\in\Z}$ are not greater than 2.

Note that the $(p-1)$th coefficient of $x^{p-1}\cdot (-x+1) \bmod{\Phi_{p}(x)}$ is 2. Substituting $x$ with $x^{p^{s-1}}$, we can see that the expansion factor of $x^k$ in $\mathcal{R}$ with $k=(p-1)\cdot p^{s-1}$ equals 2. \qed
\end{proof}

\begin{lemma} For $\mathcal{R}=\Z[x]/\Phi_{p^s q^t}(x)$, the following equality holds.
$$\max_{\substack{k \in \Z  \\ g(x)\in \mathcal{R}}}\left\{\frac{||x^k\cdot g(x)||_\infty}{||g(x)||_\infty} \right\} = 2p$$
\end{lemma}

\begin{proof}
Consider each row of the reduction matrix $R_{pq}=(I|B_1|B_2|B_3)$. The matrices $I$ and $B_2$ contain at most one nonzero element in each row. Considering the dimensions of the matrix $B_1$ and $B_3$, they contain at most $p-1$ nonzero elements in each row. 
In total, any row of $R_{pq}$ has at most $2p$ nonzero elements and they are either $-1$ or $+1$ (Corollary \ref{cor:exp_norm}).
By Remark~\ref{rem:tensor}, any row of $R_{p^s q^t}$ also has at most $2p$ nonzero elements and they are either $-1$ or $+1$.
Therefore, $||x^i \cdot g(x)||_\infty \leq 2p\cdot ||g(x)||_\infty$ holds for all $i\in \Z$, and the expansion factors of $\{x^k\}_{k\in\Z}$ are not greater than $2p$.

Combining Corollary \ref{cor:lowest_coeff_is}, \ref{cor:highest_coeff_is_one} and Lemma \ref{lem:exp_b_2}, \ref{lem:exp_b_3}, the $(\phi(pq)-1)$th row of $R_{pq}$ is of the form $[0 , \cdots , 0 , +1 | +1,  \cdots , +1 | 0 , \cdots , 0 , -1 | -1, \cdots , -1 ]$. 
Thus, the $(\phi(pq)-1)$th coefficient of $x^{\phi(pq)-1}\cdot [(1+x+\cdots+x^{p-1})-(x^q+x^{q+1}+\cdots+x^{q+p-1})] \bmod{\Phi_{pq}(x)}$ is $2p$. Substituting $x$ with $x^{p^{s-1} q^{t-1}}$, we can see that the expansion factor of $x^k$ in $\mathcal{R}$ with $k=(\phi(pq)-1)\cdot p^{s-1} q^{t-1}$ equals $2p$. \qed
\end{proof}

\section{Open Problems}
An interesting problem is to generalize the results of this paper to ternary or even to arbitrary cyclotomic polynomials.
Another direction is to investigate coefficient sizes of scaled inverses modulo cyclotomic polynomials for a wider range of polynomials than $\{x^i-x^j\}_{i,j}$.
Besides $\{x^i\}$, constructing another large subset of $\Z[x]/\Phi_M(x)$ (i) whose elements have small expansion factors (ii) and whose differences of elements have small scaled inverses is also an interesting open problem. 

\bibliographystyle{alpha}
\bibliography{mybibliography}

\end{document}